\documentclass[reqno, 12pt]{amsart}

\numberwithin{equation}{section}

\usepackage{amssymb}
\usepackage{mathrsfs}
\usepackage{dsfont}
\usepackage{enumerate, xspace}
\usepackage{changebar}
\usepackage{hyperref}
\usepackage{pdfsync}
\usepackage{color}
\usepackage{bm}
\usepackage{amsmath}

\newtheorem{thmm}{Theorem}[section]

\newtheorem{thm}{Theorem}[section]

\newtheorem{cor}[thm]{Corollary}

\newtheorem{prop}[thm]{Proposition}

\newtheorem{defin}{Definition}[section]

\newcommand\cA{{\mathcal A}}

\newcommand\cE{{\mathcal E}}
\newcommand\cF{{\mathcal F}}
\newcommand\cG{{\mathcal G}}
\newcommand\cH{{\mathcal H}}

\newcommand\cL{{\mathcal L}}

\newcommand\cS{{\mathcal S}}

\newcommand\bN{{\mathbb N}}

\newcommand\bR{{\mathbb R}}

\newcommand\eps{\epsilon}

\def\si{{\sigma}}
\def\la{{\lambda}}
\def\La{{\Lambda}}
\def\ga{{\gamma}}

\newcommand\nn{\nonumber}

 \newcommand{\dis}{\displaystyle}
 \newcommand{\und}{\underline}

\begin{document}

\title[Quasi-Static Limits]{Quasi-static hydrodynamic limits}

\author{Anna de Masi}
\address{Anna de Masi\\Universit\`a dell'Aquila,\\
 67100 L'Aquila, Italy}
\email{{\tt demasi@univaq.it}}

\author{Stefano Olla}
\address{Stefano Olla\\
CEREMADE, UMR CNRS 7534\\
Universit\'e Paris-Dauphine\\
75775 Paris-Cedex 16, France}
\email{{\tt olla@ceremade.dauphine.fr}}

\date{\today. {\bf File: {\jobname}.tex.}}
\begin{abstract}
We consider hydrodynamic limits of interacting particles systems with
open boundaries, where the exterior parameters change in a time scale
slower than the typical relaxation time scale. The limit
\emph{deterministic} profiles evolve quasi-statically. {
These limits
define rigorously the thermodynamic quasi static
transformations also for transition between non-equilibrium stationary
states}. We study first the case of the symmetric simple exclusion,
where duality can be used, and then we use relative entropy methods to
extend to other models like zero range systems. Finally we consider a
chain of anharmonic oscillators in contact with a thermal Langevin
bath with a temperature gradient and a slowly varying tension applied
to one end.  

 \end{abstract}
\thanks{We thank Lorenzo Bertini, Claudio Landim, Giovanni Jona
  Lasinio and Errico 
  Presutti for {many very helpful discussions}.
We thank the kind hospitality of GSSI, AdM thanks also the Institute
H. Poincare, where part of this work has been done. 
The work of S.O. has been partially supported by the European
Advanced Grant Macroscopic Laws and Dynamical Systems (MALADY) (ERC
AdG 246953)} 
\keywords{Quasi-Static thermodynamic, hydrodynamic limits, Clausius equality}
\subjclass[2000]{...}
\maketitle

\section{Introduction}
\label{sec:introduction}

Quasi-static transformations are defined in thermodynamic literature
as those transformations where the external conditions change so
slowly that at any moment the system is \emph{approximately} in
equilibrium. They are usually presented as idealization of 
\emph{real} thermodynamic transformations \cite{callen}. 
On the other hand they are necessary, in the traditional approach, in
order to construct thermodynamic potentials, for example to define 
{thermodynamic} entropy from Carnot cycles.  

Equilibrium statistical mechanics succesfully describes the
equilibrium thermodynamics states from the microscopic dynamical
properties of the     
system, in particular the corresponding stationary probability
distributions, given explicitely by the Gibbs measures.   
Let us consider the simple situation when the internal dynamics of the 
system has a certain number of conserved quantities (energy,
density,...), and the external conditions (temperature of the heat
bath, pressure of the external forces,...) select the stationary
equilibrium probability distribution for the microscopic dynamics. The
fact that, when the microscopic system is \emph{large}, there are only
few of such conserved quantities, and correspondingly few
external parameters will determine the equilibrium state, is related to
the ergodicity of the microscopic dynamics. 

The system can be driven out of equilibrium by changing the exterior
parameters. If this is done abruptly, eventually the system will relax
to a new equilibrium state corresponding to the new values of
temperature, pressure etc.. {
The typical relaxation time to equilibrium may depend on the size of
the microscopic system}.

The situation is similar but  more complex when the external agents
are space 
inhomogeneous, for example when the system is submitted to gradients of
temperature or pressure. In the stationary situation fluxes of the
conserved quantities can cross the system. The corresponding
probability stationary distributions for the microscopic dynamics are
called \emph{non-equilibrium stationary states}. These are not
explicitly computable as in the equilibrium case, but are locally
close to the equilibrium ones, when the size of the system is very large. 

In the present article we want to obtain, at least for some systems,
in a rigorous way the quasi-static thermodynamic transformations from
the microscopic dynamics, through a space-time scale limit. In systems
of size $N$ 
with diffusive behavior, where the typical time scale for reaching
equilibrium is of order $N^2$, we look at the evolution of the
conserved quantities at the time scale $N^{2+\alpha}t$, for a $\alpha>0$,
and we change  the exterior conditions smoothly at this time scale
(i.e. very slowly with respect to the typical relaxation time). We
obtain that in this time scale, at each instant $t$, the profile of the
conserved quantities converges to the deterministic profiles
corresponding to the stationary state for the boundary condition given
at time $t$. This means that the profiles evolve quasi-statically,
following 
the rules of the thermodynamic quasi-static transformations.
Consequently the \emph{paradox} of the quasi-static transformations
disappears when these transformations are understood (also in a
mathematically precise way) in a proper macroscopic space-time scale.   
Remarkably we also obtain quasi-static transformations
between \emph{non-equilibrium stationary states}, that could be the
base for constructing a thermodynamics of non-equilibrium.    

We consider three different models, in order to illustrate the
generality of the idea. 

We consider first the symmetric simple exclusion in an interval of
size $N$, i.e. particles performing symmetric random walks with
exclusion rule, with density reservoirs at the boundaries. These
reservoirs are realized with random creation and annhiliation of
particles {so that the
average density of particles  at the boundaries} will be respectively
$\rho_-(t)$ on the left and  $\rho_+(t)$ on the right. The rate
of jumps or of creations/annhiliation is  $N^{2+\alpha}$, {$\alpha$ being any positive number}. We prove that 
at every time $t$ the empirical  {particle's} density 
converges, as
$N\to\infty$, to the linear interpolation between $\rho_-(t)$ and
$\rho_+(t)$, i.e. to the solution of the Laplace equation $\partial_{xx}
\rho = 0$ with the given boundary conditions. {We actually prove the strongest result of  the local equilibrium property.


We first prove  the above convergence in section \ref{sec:simple-excl-with}
   by using duality techniques as in \cite{GKMP},  where analogous results were proved in the case of $\rho_\pm$ fixed constants.
With the duality techniques we prove local equilibrium in the strong form of Theorem \ref{thm1} and the same techniques allow to study the limit of
the covariance function. In fact we prove in Theorem \ref{a3.2}  that the covariance behaves at each
time $t$ as in the corresponding stationary states}, as first
computed by Spohn \cite{Spohn}.

{In Section \ref{sec:entropy-method}, see Theorem \ref{leex},
we  prove a weaker version of the local equilibrium by using} entropy methods as in \cite{GPV} and \cite{KOV1989}
 that
are suitable to be used for more general
models.   {In particular the proof of Theorem \ref{leex} can be adapted to prove the analogous statement for all gradient conservative dynamics with creations/annihilations at the boundaries as the ones studied in \cite{els1} in the case of constants $\rho_\pm$.}

{In Section \ref{sec:zero-range} we} consider the zero-range process,
where more particles per site 
are allowed, and jump symmetrically with rate depending on the number
of particles in the site they occupy. A chemical potential $\lambda$
characterizes the equilibrium measures, which is the expectation of the
jump rate, and in general is a non-linear function of the density. At
the boundaries the reservoirs are given again by creations/annihilations
of particles corresponding to two possibly different values of
$\lambda_{\pm}(t)$ at time $t$. {Jumps and creations/annihilations
happen at rate proportional to $N^{2+\alpha}$}, and  we prove that at
each macroscopic 
time $t$ (in a weak sense) the empirical density of the particles converges
to the solution of the non-linear Laplace equation
\begin{equation}
  \label{eq:macro0r}
  \partial_{xx} \lambda(\rho(x,t)) = 0, \qquad \lambda(\rho(\pm 1,t))
  = \lambda_{\pm}(t). 
\end{equation}

Finally, {in Section \ref{sec:damp-anharm-chain}}, we consider a chain of (unpinned) anharmonic oscillators in contact with 
a heat bath with a gradient of temperature and subject to a force
(tension) $\bar\tau(t)$ acting on the last right particle, that changes at
the {macroscopic 
time scale}. The left end of the chain is
attached to a point and the heat bath is modeled by Langevin
thermostats acting on each particles. The temperatures of the
Langevin thermostats may depend on the atom and change slowly
form one atom to the neighbor one, giving rise to a smooth
macroscopic profile of temperature.  When temperature profile
$\beta^{-1}(y)$  is constant in space and time at a value $\beta^{-1}$
and tension is constant $\bar\tau(t) = \tau$, the stationary
(equilibrium) distribution is given 
by the canonical Gibbs measure with these values of temperaure and
tension. In equilibrium the tension $\tau$ is equal to the expectation
value of the anharmonic force between particles.

For a non-constant profile of temperature we have a
non-equilibrium stationary state where energy flows from hot to cold
thermostats. We consider only situations where the profile of
temperature does not change in time, and only the applied tension is
changing {
 in the macroscopic 
 time scale}. Again the dynamics is speeded up in time by $N^{2+\alpha}$. 
We prove that the empirical strain of volume converge to the solution
of the equation 
\begin{equation}
  \label{eq:macrossillatori}
  \begin{split}
    \tau(r(x,t), \beta(x)) = \bar\tau(t)
  \end{split}
\end{equation}
where $\tau(r,\beta)$ is the equilibrium tension corresponding to the
volume (lenght) $r$ and temperature $\beta^{-1}$, i.e. $\tau
= \partial_r \mathcal F(r,\beta)$, where  $\mathcal F(r,\beta)$ is the
thermodynamic free energy. We have obtained this way the quasi-static
isothermal transformation. 

The main interest in this last model is that we can define the heat
$Q(t)$ as the limit of the the total flux of energy between the system
and the thermostats, divided by the number of particles, and it turns
out  that this is a \textbf{deterministic} function of time that satisfy
\begin{equation}
  \label{eq:1stintro}
  U(t) - U(0) = Q(t) + \mathcal W(t)
\end{equation}
where $U(t)$ is the limit of the internal energy (per particle) and
$\mathcal W(t)$ is the work done by the exterior force. In the context
of the equilibrium case (constant temperature profile), this is the
first principle of thermodynamics for isothermal quasi-static
transformations. For non-constant temperature profile $Q(t)$ is
usually called \emph{excess heat}.  
Also this quasi-static transformation satisfies Clausius identity,
i.e. the work $\mathcal W(t)$ is equal to the time difference of the
free energy after properly identifying the macroscopic free energy as
integral over space of the equilibrium one.  

Quasi-static limits where previously considered on the macroscopic
hydrodynamic diffusive equation, rescaling the boundaries parameters
after having performed the hydrodynamic limits (cf. \cite{bgjll13},
\cite{olla2014}, \cite{LO2015}), i.e. in a two step approach. The point
of this article is to show that quasi-static transformations can be
obtained in a straightforward way {from the microscopic dynamics by
  considering} 
the proper space-time scaling.

\section{Simple exclusion with boundaries}
\label{sec:simple-excl-with}

We consider the exclusion process in $\{0,1\}^{\La_N}$,
$\La_N:=\{-N,..,N\}$ with reservoirs at the boundaries with density  
$\rho_\pm(t)\in (0,1)$. We assume that $\rho_\pm(t)$ are Lipschitz
continuous. {Presumably our results can be generalized to piecewise
  continuous functions but we did not investigate this case.} 

Denoting by $\eta(x)\in  \{0,1\}$ the occupation number at $x\in
\La_N$ we define the dynamics via the generator 
		\begin{equation}
		\label{0.1}
 L_{N,t}=N^{2+\alpha}[L_{\text{exc}}+L_{b,t}],\qquad t\ge 0,\quad \alpha>0
			\end{equation}
where {
 \begin{equation}
   \label{0.2}
L_{\text{exc}}f(\eta) 
=\frac 12  \sum_{x=-N}^{N-1} \Big(f(\eta^{(x,x+1)})-f(\eta)\Big)=:
\frac 12  \sum_{x=-N}^{N-1}\nabla_{x,x+1}f(\eta)
   \end{equation}
$\eta^{(x,y)}$ is the configuration obtained from $\eta$ by exchanging the occupation numbers at $x$ and $y$}, and  
\begin{eqnarray}
L_{b,t}f(\eta) = \sum_{\sigma = \pm} \rho_{\sigma}(t)^{1-\eta(\sigma N)}
  (1-\rho_{\sigma}(t))^{\eta(\sigma N)}  [f(\eta^{\sigma N})-f(\eta)]
 \label{eq:1}
\end{eqnarray}
where $\eta^x (x) = 1-\eta(x)$, and $\eta^x (y) = \eta(y)$ for $x\neq y$.

\bigskip

	\begin{thm}
	\label{thm1}
For any $\alpha>0$ and  for  any macroscopic time $t> 0$ the following holds. For any initial configuration $\eta_0$, for any $r\in[-1,1]$ and for any local function $\varphi$
	\begin{equation}
	\label{0.5}
\lim_{N\to\infty}\mathbb E_{\eta_0}\Big(
\theta_{[Nr]} \varphi (\eta_{t})\Big)=
 <\varphi(\eta)>_{\bar \rho(r,t)} =: \hat\varphi(\rho(r,t)) 
	\end{equation}
where $[\cdot]$ denotes integer part, $\theta$ is the shift operator, $<\cdot>_{\rho}$ is the expectation with respect to the product Bernoulli measure of density $\rho$, 
and  
\begin{equation}
 \label{eq:2} 
 \bar \rho(r,t)=\frac 12[\rho_+(t)-\rho_-(t)]r +\frac 12\left[\rho_+(t) + \rho_-(t)\right], \qquad r\in[-1,1]
 \end{equation}
is the quasi-static profile of density at time $t$.
	\end{thm}

We prove Theorem \ref{thm1} by using the self-duality of the exclusion
that can be adapted to our process as we explain in Propositions
\ref{prop1.1} and \ref{prop1.2} below. The dual process is defined in
terms {of absorbed stirring walks  
as explained in Definition \ref{def1} below.} 

\medskip
In the rest of this section we use the equivalent definition of the
dynamics obtained by considering the generator 
	\begin{equation}\label{0.3}
 \hat L_{N,s}:=L_{\text{exc}}+L_{b,N^{-2-\alpha}s}
	\end{equation}
and then studying the process up to times of order $N^{2+\alpha}t$, $t\ge 0$ being the macroscopic time.

	\begin{defin} {The absorbed walkers}.
	\label{def1}

	\noindent
	One absorbed random walk is the Markov process in \newline
        $\La_{N+1}:=[-N-1,N+1]\cap \mathbb Z$ that after an
        exponential time of mean 1 jumps with equal probability on its
        nearest neighbor sites and when it reaches $\pm(N+1)$ it
        stays there forever.  We denote by $\{x_t, t\ge 0\}$ a
        trajectory and we denote by $\tau$ the absorbing time: 
	\begin{equation}
	\label{tau}
\tau=\min\{\tau_{-(N+1)},\tau_{N+1}\},\qquad \tau_a=\inf\{t\ge 0: x_t=a\}
	\end{equation}

For any positive integer $n$ we also consider the process
$\{(x^{(1)}_t,.,x^{(n)}_t)$, $t\ge 0\}$ of $n$ {stirring
  walks absorbed at  $\pm(N+1)$. The stirring process is defined by
  exchanging independently at rate $\frac 12$ the ``content" in $x$
  and $x+1$, so that if $x^{(i)}_t=x$ before the exchange then  it
  moves to $x+1$ while if  $x^{(j)}_t=x+1$ then it moves to $x$ after
  the exchange (see \eqref{genS} below for the generator of the
  stirring process). 
 If one disregards labels this is the exclusion process.} 
We denote by $\{x^{(i)}_s, s\ge 0\}$, $i=1,.,n$ the trajectory of particle $i$ starting from $x^{(i)}_0=x_i$ and call $\tau_{i}$ its absorption time. We call $\mathcal L$ the generator,
$\mathcal P_{\und x}$ the law and  $\mathcal E_{\und x}$ the expectation  of this process starting from $\und x=(x_1,.,x_n)$.

 		\end{defin}
%

 	\begin{prop}
	\label{prop1.1}
Let $\eta_0$ be any initial configuration. Then for all $x\in\La_N$
	\begin{equation}
\mathbb E_{\eta_0}[\eta_t(x)]=\mathcal E_x\big[\eta_0(x_{t})\mathbf 1_{\tau>t}\big]+
\mathcal E_x\big[\mathbf 1_{\tau\le t}\,\,u_{x_{\tau}}(t-\tau)\big]
	\label{1.0}
	\end{equation}
where
\begin{equation}
	\label{1.5}
u_{\pm(N+1)}(s)=\rho_\pm(N^{-(2+\alpha)}s)
	\end{equation}
	\end{prop}
\noindent Below we often write $u_{\pm}(s)=u_{\pm(N+1)}(s)$.

{\bf Proof.} Given  $x$ and $t$, a trajectory ${\und {\varkappa}}=\{x_s,
s\in [0,t]\}$ of the absorbed random walk starting at $x$ and a
trajectory $\und \eta=\{\eta_s, s\in [0,t]\}$ starting from $\eta_0$,
we define for all $s\le t$  
			\begin{equation}
	\label{dual}
\psi\big(\und {\varkappa},\und \eta, s\big)=\begin{cases}
 \eta_s(x_{t-s}),& \text{if }\tau>t-s, \\
   u_{x_{\tau}}(t- \tau) & \text{if } \tau\le t-s \\
   \end{cases} 
	\end{equation}
Calling $P=\mathcal P_x\times \mathbb P_{\eta_0}$ the (product) law of 
$({\und {\varkappa}},\und \eta)$ and $E$ the expectation with respect
to $P$, we observe that for $x\in\La_N$ 
	\begin{equation}
	\label{1.3}
E(\psi\big(\und {\varkappa},\und \eta, t\big))=\mathbb
E_{\eta_0}(\eta_t(x)),\qquad E(\psi\big(\und {\varkappa},\und \eta,
0\big))= \text{ right hand side of \eqref{1.0}} 
			\end{equation}
Thus to prove \eqref{1.0} we  show below that
$\frac{d}{ds}E\big[\big(\psi({\und {\varkappa}},\und \eta, s)\big)\big]=0$.   
We first define
for $x\in\La_{N+1}$, $\eta\in \{0,1\}^{\La_N}$ and $a_\pm\in(0,1)$
		\begin{equation}
	\label{dual2}
\varphi_{a_\pm}\big( x,\eta\big)=\begin{cases}
 \eta(x),& \text{if } |x|\le N, \\
   a_\pm & \text{if } x=\pm (N+1) \\
   \end{cases} 
	\end{equation}
and we will prove that 
	\begin{equation}
	\label{1.7}
	\frac{d}{ds}E\big(\psi({\und {\varkappa}},\und \eta,
        s)\big)={\Bigg(\frac{d}{ds} 
	E\big(\varphi_{a_\pm}\big(
        x_{t-s},\eta_{s}\big)\big)\Bigg)}\Big|_{a_\pm=u_\pm(t-s)} 
\end{equation}
Let	$0\le s< s'\le t$ and call
	\begin{eqnarray}
	\nn
	&&D(s,s')=
\psi({\und {\varkappa}},\und \eta, s')-\psi({\und {\varkappa}},\und
           \eta, s),\\
&& \tilde D(s,s')=\varphi_{a_\pm}\big(
   x_{t-s'},\eta_{s'}\big)-\varphi_{a_\pm}\big( x_{t-s},\eta_{s}\big) 
	\end{eqnarray}
First observe that for any choice of $a_\pm$,
\begin{eqnarray}\label{1.8}
 D(s,s')\mathbf 1_{\tau\le t-s'}=0,\qquad \tilde D(s,s')\mathbf 1_{\tau\le t-s'}=0	
 	\end{eqnarray}
If instead $\tau> t-s'$ it is important to choose $a_\pm=u_\pm(t-s)$. In fact, since $u_\pm$ are Lipschitz continuous functions,   we have
	\begin{equation}
	\label{1.6}
\sup_{\si\in[s,s')}|\psi\big({\und {\varkappa}},\und \eta, \si\big)-
\varphi_{u_\pm(t-s)}\big( x_{t-\si},\eta_{\si}\big) |\mathbf
1_{\tau>t-s'}\le C (s'-s)\mathbf 1_{\tau\in[t-s', t-s]} 
	\end{equation}
{Since $P(\tau\in[t-s', t-s])\to 0$ as $s'\downarrow s$, from 
from \eqref{1.8} and \eqref{1.6} we get that 
	\begin{equation*}
\lim_{s'\downarrow s}\frac {E(D(s,s'))}{s'-s}=\Bigg(\lim_{s'\downarrow s}\frac {E(\tilde D(s,s'))}{s'-s})\Bigg)\Big|_{a_\pm=u_\pm(t-s)}
	\end{equation*}
}
Thus the Proposition will follows from the fact that the derivative of $E(\varphi_{a_\pm})$ is 0. This last derivative is easier to compute because $\varphi_{a_\pm}$  is a function of $(x,\eta)$ unlike $\psi$ that is a function of the whole trajectories $({\und {\varkappa}}, \und \eta)$.  In fact we have
	\begin{eqnarray}
	\nn
&&\hskip-.8cm\frac d{ds} E\big[\varphi_{a_\pm}\big(x_{t-s},\eta_s\big)\big]= - E\big[\mathcal L\varphi_{a_\pm}\big(x_{t-s},\eta_s\big)\big]
+E\big[ \hat L_{N,s} \varphi_{a_\pm}\big(x_{t-s},\eta_s\big)\big]
\\&&=-[a_+-E\big(\eta_s(N)\mathbf
     1_{x_{t-s}=N}\big)]+[u_+(t-s)-E\big(\eta_s(N)\mathbf
     1_{x_{t-s}=N}\big)]\nn 
\\&&-[a_--E\big(\eta_s(-N)\mathbf
     1_{x_{t-s}=-N}\big)]+[u_-(t-s)-E\big(\eta_s(-N)\mathbf
     1_{x_{t-s}=-N}\big)] 
\nn \\\label{1.9}
	\end{eqnarray}
because the contribution of the generator $L_{\rm exc}$ cancels with the jumps inside $[-N,N]$ of the random walk.  Thus $\frac d{ds} E\big[\varphi_{a_\pm}\big(x_{t-s},\eta_s\big)\big]=0$ by putting in \eqref{1.9}  $a_\pm=u_\pm(t-s)$ in agreement with \eqref{1.7}. \qed
	
	\begin{cor}
	\label{cor1}
For any $\alpha>0$, any initial configuration $\eta_0$, any $t>0$ and any $r\in[-1,1]$
	\begin{equation}
	\label{1.11}
\lim_{N\to\infty}\,\,
\mathbb E_{\eta_0}\big[\eta_{N^{2+\alpha}t}([Nr])\big] \ =\ \bar\rho(r,t)
	\end{equation}
	\end{cor}

	
\noindent{\bf Proof.}  We use \eqref{1.0}. First observe that for any $x\in\La_N$
	\begin{equation}
	\label{1.12}
\mathcal P_x(\tau>s)\le c e^{-c' s/N^2}
	\end{equation}
thus the first term on the right hand side of \eqref{1.0} converges to 0 exponentially. Call 		\begin{equation*}
F_{N,x}(s)=\mathcal P_x\big[\tau_{N+1}\le s, \tau_{-N-1}>\tau_{N+1}\big]
\end{equation*}
 then   by \eqref{1.12}
	\begin{eqnarray}
	\nn
&&\hskip-.6cm
\big|\mathcal E_x\big[\mathbf 1_{\tau=\tau_{N+1}\le N^{2+\alpha}t}\,\,u_{N+1}(N^{2+\alpha}t-\tau)\big]-\int_0^{N^{2+\frac\alpha 2}t}
\,u_{N+1}(N^{2+\alpha}t-s)\,dF_{N,x}(s)
\big|
\\&&\hskip6cm \nn\le
{ c e^{- c' N^{\alpha/2}}}
\end{eqnarray}
Recalling \eqref{1.5} we next observe that by Lipschitz continuity
\begin{eqnarray}
	\nn
&&\nn
\int_0^{N^{2+\alpha/2}t}
\,\Big|u_{N+1}(N^{2+\alpha}t-s)-\rho_+(t)\Big|dF_{N,x}(s)\le c\frac {tN^{2+\alpha/2}}{N^{2+\alpha}}\le c\frac t{N^{\alpha/2}}
	\end{eqnarray}
	Finally, by \eqref{1.12},
	\begin{eqnarray}
	\nn|F_{N,x}( N^{2+ \alpha/2}t)- \mathcal P_x\big[\tau_{-(N+1)}>\tau_{N+1}\big]\Big|\le 
\mathcal P_x\big[\tau_{N+1}>N^{2+ \alpha/2}t\big]\le ce^{-c'N^{\alpha/2}}
	\end{eqnarray}
and analogously for the term with $u_{-(N+1)}$. Finally observe that
	\begin{eqnarray*}
&& \mathcal P_x(\tau_{N+1}\le\tau_{-(N+1)})=\frac 12+\frac x{2(N+1)}
   \end{eqnarray*}
   \qed

\bigskip
We now generalize \eqref{1.0} by writing a duality formula for the correlation functions. We thus consider $n>1$ distinct points $x_1$,..,$x_n$  in $\La_N$ and recalling Definition \ref{def1} we prove the following Proposition.

%
%
%
	\begin{prop}
	\label{prop1.2}
For any positive integer $n$, any $x_1$,..,$x_n$ distinct points in $\La_N$ and for any $\eta_0$ we have that for all $t\ge 0$ (recall \eqref{1.5})
	\begin{equation}
	\label{1.14}
\mathbb E_{\eta_0}\big[\prod_{i=1}^n\eta_t(x_i)\big]=\mathcal E_{\und x}\Big[\prod_{i=1}^n\big[\mathbf 1_{\tau_i>t}\,\,\eta_0(x^{(i)}_t)+\mathbf 1_{\tau_i\le t}\,\,u_{x^{(i)}_{\tau_i}}(t-\tau_i)\big]\Big]
	\end{equation}

	\end{prop}

\noindent{\bf Proof.} Denoting by $E=\mathcal P_{\und x}\times\mathbb E_{\eta_0}$ and by $\und X=(x^{(1)}_s,.., x_s^{(n)}, s\in[0,t])$, we prove below that
	\begin{equation}
	\label{1.19}
\frac d{ds}E\big[\Psi(\und X,\und\eta,s)\big]=0,
\qquad \Psi(\und X,\und\eta,s)=\prod_{i=1}^n\psi(\und x^{(i)},\und \eta,s)
	\end{equation}
The strategy is to adapt the proof given in  Proposition \ref{prop1.1} for $n=1$.  Let 
$\mathcal F_{t-s}$ be the sigma-algebra generated by $\{(x^{(1)}_{s'},,,x^{(n)}_{s'}), s'\in[0,t-s]\}$. Given $\mathcal F_{t-s}$ we know $k$= number of particles absorbed by time $t-s$ and the times of absorption $\tau_i<t-s$, $i=i_1,..i_k$. Thus calling
	\begin{equation*}
U(t-s):= \prod_{j\in I}u_{x^{(j)}_{\tau_j}}(t-\tau_j)\mathbf 1_{\tau_j\le t-s},\qquad I=\{i_1,..i_k\}
	\end{equation*}
we have that for all $0\le s'<s\le t$,
	\begin{eqnarray*}
&&\hskip-.5cm
E\big[\Psi(\und X,\und\eta,s')-\Psi(\und X,\und\eta,s))\big]
\\&&\hskip1cm
=E\Big[U(t-s)
E\big[\prod_{i\notin I}\mathbf 1_{\tau_i>t-s}(\psi(\und x^{(i)},\und \eta,s')-
\psi(\und x^{(i)},\und \eta,s))\big|\mathcal F_{t-s}\big]\Big]
	\end{eqnarray*}
Recalling \eqref{dual2}, as 
in \eqref{1.6} we have that for all $0\le s'<s\le t$
		\begin{eqnarray*}
&& \hskip-.5cm 
\sup_{\si\in[s',s)}\Big|\prod_{i}\mathbf 1_{\tau_i>t-s}\psi(\und x^{(i)},\und \eta,\si)-\prod_i\mathbf 1_{\tau_i>t-s}\varphi_{u_\pm(t-s)}( x^{(i)}_{t-\si},\eta,\si)
\Big| \\&&\hskip1cm
\le C(s'-s)\mathbf 1_{\exists i :
\tau_i\in[t-s, t-s']}
		\end{eqnarray*}
Thus as in the case $n=1$,  \eqref{1.14} is a consequence of
	\begin{eqnarray*}
&&\lim_{s'\nearrow s}\frac 1{s-s'}E\Big[U(t-s)
E\big[\prod_{i\notin I}\mathbf 1_{\tau_i>t-s}\\&&\hskip1cm
(\varphi_{u_\pm(t-s)}(x^{(i)}(t-s'), \eta_{s'})-
\varphi_{u_\pm(t-s)}(x^{(i)}(t-s), \eta_s))\big|\mathcal F_{t-s}\big]\Big]=0
	\end{eqnarray*}
The above equality can be proved by applying the generators similarly to \eqref{1.9}.\qed
	\begin{cor}
	\label{cor2}
For any initial configuration $\eta_0$, for any $t>0$, any $n>1$ and any distinct points $x_1$,..,$x_n$  in $\La_N$, calling $t_N=N^{2+\alpha}t$
	\begin{equation}
	\label{1.19}
\lim_{N\to\infty}\Big| \mathbb E_{\eta_0}\big[\prod_{i=1}^n\eta_{t_N}(x_i)\big]-
\prod_{i=1}^n	\mathbb E_{\eta_0}\big[\eta_{t_N}(x_i)\big]\Big|=0
	\end{equation}
	\end{cor}
	
\noindent{\bf Proof.} By \eqref{1.14} and  \eqref{1.12} it is enough to prove that
	\begin{eqnarray}
	\label{1.20}
&&\lim_{N\to\infty}|D_{\und x,N}|=0 \end{eqnarray}
where
		\begin{eqnarray}
\nn D_{\und x,N}= \mathcal E_{\und x}\big[\prod_{i=1}^n\mathbf 1_{\tau_i\le t_N}\,\,u_{x^{(i)}_{\tau_i}}(t_N-\tau_i)\big]-
\prod_{i=1}^n	\mathcal E_{ x_i}\big[\mathbf 1_{\tau_i\le t_N}\,\,u_{x^{(i)}_{\tau_i}}(t_N-\tau_i)\big]
	\end{eqnarray}

We use the coupling explained in \cite{anna-errico91} (see also \cite{DPTV2}) between $n$ independent random walks and $n$  stirring particles. Both the stirring
and the independent particles move in the whole $\mathbb Z$. 
{The evolution of the $n$ stirring particles is the process with generator $L^{\rm st}$ that acts on functions $f$ on $\mathbb Z^n_{\ne}=\{\und x\in \mathbb Z^n: x_i\ne x_j$ for all $i\ne j\}$ as follows
	\begin{equation}
	\label{genS}
	L^{\rm st}f(\und x)=\frac 12\sum_{i=1}^n \sum_{\si=\pm 1}[f(\und x^{i,\si})-f(\und x)]
	\end{equation}
with $\und x^{i,\si}=\und x+\si e_i-\si e_j$ if there is $j\ne i$ such that $x_j= x_i+\si e_i$, otherwise  $\und x^{i,\si}=\und x+\si e_i$, $e_i$ being the unit vector in the $i$ direction.
We refer to Section 6.6 of \cite{anna-errico91} for the definition of the coupling and we denote by $x^{0}_i(t)$  and $x_i(t)$  the position at time $t$ of the 
independent and respectively interacting particle with label $i$.}
The initial positions are the same $\und x^0(0)=\und x(0)$ and
the law of the coupling is denoted by $Q_{\und x}$ (we denote by
$Q_{\und x}$ also the expectation).  In Proposition 6.6.3 of
\cite{anna-errico91}  it is proved that for any positive $\ga$ and $k$
there is $c$ so that for all $t>0$ 
{
	\begin{equation}
	\label{1.21}
Q_{\und x}\left(\sup_{0<s\le t}|x_i(s)-x_i^0(s)|\le t^{\frac
  14+\ga},\,\,\forall i=1,..,n\right)\ge 1-ct^{-k} 
	\end{equation}}
We call $\tau^0_i=\min(\tau^{0,-}_i,\tau^{0,+}_i)$ with $\tau_i^{0,\pm}$ the first time the independent particle $i$ is at $\pm (N+1)$, 
$\tau_i=\min(\tau_i^+,\tau_i^-)$ 
is the analogue hitting time for the interacting particle $i$. Observe that since the particles are moving by stirring, then the hitting time of interacting particle $i$ is not influenced by the fact that another particle (say $j$) is inside the interval after being outside. Moreover the function inside the expectation in $D_{\und x,N}$ depends only on the hitting times. 	Thus 
		\begin{eqnarray}
	\nn 
&&\hskip-1cm
D_{\und x,N}=Q_{\und x}\Big(\prod_{i=1}^n\big[\mathbf 1_{\tau_i\le t_N}\,\,u_{x_i(\tau_i)}(t_N-\tau_i)\big]
\\&&\hskip2cm  \nn -
\prod_{i=1}^n	\big[\mathbf 1_{\tau^0_i\le t_N}\,\,u_{x^0_i(\tau^0_i)}(t_N-\tau^0_i)\big]\Big)
	\label{1.22}
	\end{eqnarray}
We call $\und a=(a_1,..,a_n)$, $a_i\in\{-1,1\}$ the vector that specifies where the stirring particles exit the interval, namely $a_i=\pm 1$ if stirring particle $i$ exits in $\pm (N+1)$, thus $\tau_i=\tau_i^{a_i}$. Analogously $\und b=(b_1,..,b_n)$, $b_i\in\{-1,1\}$ specifies that $\tau_i^0=\tau^{0,b_i}_i$. Given $\und a$ and $\und b$ we denote by $\tau(\und a)=(\tau_1^{a_1},..,\tau_n^{a_n})$ and by $\tau(\und b)=(\tau_1^{b_1},..,\tau_n^{b_n})$. With this notation we have
		\begin{eqnarray}
		\label{1.23}
&&D_{\und x,N}=\sum_{\und a}\sum_{\und b}Q_{\und x}\big(\phi(\tau(\und a),\tau(\und b))\big)\end{eqnarray}
		\begin{eqnarray*}
&&\hskip-1.4cm
\phi(\tau(\und a),\tau(\und b))=\prod_{i=1}^n\big[\mathbf 1_{\tau_i=\tau_i^{a_i}\le t_N}\,\,u_{a_i}(t_N-\tau_i)\big]-
\\&&\hskip2cm
\prod_{i=1}^n	\big[\mathbf 1_{\tau^0_i=\tau^{0,b_i}_i\le t_N}\,\,u_{b_i}(t_N-\tau^0_i)\big]\nn
	\end{eqnarray*}
	Call
		\begin{eqnarray*}
A_N=\{\tau_i<N^{2+\alpha/2}, \tau^0_i<N^{2+\alpha/2}\text{ for all } i=1,..,n\}
		\end{eqnarray*}
Then by \eqref{1.12} 
		\begin{eqnarray}
		\label{1.24}
Q_{\und x}\big(\phi(\tau(\und a),\tau(\und b))\big)=Q_{\und x}\big(\phi(\tau(\und a),\tau(\und b))\mathbf 1_{A_N}\big)+c'e^{-c N^{\alpha/2}}
		\end{eqnarray}
We now use the Lipschitz continuity of $u_\pm$ 
		\begin{eqnarray*}
&&|Q_{\und x}\big(\phi(\tau(\und a),\tau(\und b))\mathbf 1_{A_N}\big)-Q_{\und x}\big(\Phi(\tau(\und a),\tau(\und b))\mathbf 1_{A_N}\big)|\le c''\frac1 {N^{a/2}}
\\&& \Phi(\tau(\und a),\tau(\und b))=\prod_{i=1}^n u_{a_i}(t_N)\mathbf 1_{\tau_i=\tau_i^{a_i}\le t_N}\,\, -
\prod_{i=1}^n	u_{b_i}(t_N) \mathbf 1_{\tau^0_i=\tau^{0,b_i}_i\le t_N} 
	\nn	\end{eqnarray*}
Observe now that if $\und a=\und b$ then 	$ \Phi(\tau(\und a),\tau(\und b))=0$, thus
	\begin{eqnarray}
	\nn
|Q_{\und x}\big(\Phi(\tau(\und a),\tau(\und b))\mathbf 1_{A_N}\big)\le C \sum_{i=1}^n\sum_{b\in\{-1,1\}}Q_{\und x}\big(\mathbf 1_{\tau_i=\tau_i^{-b}\le t_N}\,\, -
 \mathbf 1_{\tau^0_i=\tau^{0,b}_i\le t_N} \big)\\	\label{1.26}
	\end{eqnarray}
	Call
		\begin{eqnarray*}
B_N=\{|x_i(t)-x^0_i(t)|<N^{\frac 12+\frac{\alpha}4}\text{ for all } t\in[0,N^{2+\alpha/2}]\}
		\end{eqnarray*}
thus by \eqref{1.21} for a suitable $\ga$ and for $k$ {sufficiently large}
	\begin{eqnarray}
\nn
&&\hskip-1cm
\Big|	\sum_{b\in\{-1,1\}}Q_{\und x}\big(\mathbf 1_{\tau_i=\tau_i^{-b}\le t_N}\,\, -
 \mathbf 1_{\tau^0_i=\tau^{0,b}_i\le t_N} \big)\\&&\nn\hskip1cm-
Q_{\und x}\big(\mathbf 1_{B_N}\sum_{b\in\{-1,1\}}\big[ \mathbf 1_{\tau_i=\tau_i^{-b}\le t_N}\,\, - \mathbf 1_{\tau^0_i=\tau^{0,b}_i\le t_N}\big] \big)\Big|
 \le cN^{-k}\\	\label{1.27}
 	\end{eqnarray}
Call $X^\pm_N=N+1\mp N^{\frac 12+\frac{\alpha}2}$ and $Y^\pm_N=-N-1\mp N^{\frac 12+\frac{\alpha}2}$. Call
\begin{equation*}
  \begin{split}
    C_N= \left\{\text{ independent particle $i$ reaches $Y_N^+$ and
        goes to $X_N^-$ before reaching $Y_N^-$}\right\} \\
    \bigcup
    \left\{\text{independent particle $i$ reaches $X_N^-$ and goes to
        $Y_N^-$ before reaching $X_N^+$}\right\}.
  \end{split}
\end{equation*}
 Then in $C^c_N$, the complement set of $C_N$, we have that independent and interacting particles $i$ exit from the same side. Thus {
 	\begin{eqnarray*}
	Q_{\und x}\big(\mathbf 1_{C^c_N}\mathbf 1_{B_N}\sum_{b\in\{-1,1\}}\mathbf 1_{\tau_i=\tau_i^{b}\le t_N}\big)=Q_{\und x}\big(\mathbf 1_{C^c_N}\mathbf 1_{B_N}\sum_{b\in\{-1,1\}} \mathbf 1_{\tau_i=\tau_i^{b}\le t_N} \mathbf 1_{\tau^0_i=\tau^{0,b}_i\le t_N}\big)
		\end{eqnarray*}
so that}
			\begin{eqnarray*}
	Q_{\und x}\big(\mathbf 1_{C^c_N}\mathbf 1_{B_N}\sum_{b\in\{-1,1\}}\big[ \mathbf 1_{\tau_i=\tau_i^{-b}\le t_N}\,\, -
 \mathbf 1_{\tau^0_i=\tau^{0,b}_i\le t_N}\big] \big)	=0
		\end{eqnarray*}
We are thus left with the trajectories in $C_N$ which is an event that involves only the independent particle $i$. Thus
		\begin{eqnarray*}
&&	
\hskip-1.7cm
Q_{\und x}\big(\mathbf 1_{C_N}\mathbf 1_{B_N}\sum_{b\in\{-1,1\}}\big[ \mathbf 1_{\tau_i=\tau_i^{-b}\le t_N}\,\, -
 \mathbf 1_{\tau^0_i=\tau^{0,b}_i\le t_N}\big] \big)	\le Q_{\und x}\big(C_N\big)
\\&&
\le \mathcal P^0_{Y^+_N}(\tau_{X^-_N}>\tau_{Y^-_N})+
 \mathcal P^0_{X^-_N}(\tau_{Y^+_N}>\tau_{X^+_N})\le {c}\,\,\frac
     {N^{\frac 12+\frac\alpha 2}}{2(N+1)} 
		\end{eqnarray*}
		\qed
		
\begin{proof}[ Proof of Theorem \ref{thm1}] Since any cylinder
  function is a linear combination {of functions of the type
    $\prod_x\eta(x)$,} 
\eqref{0.5} is a consequence of \eqref{1.11} and \eqref{1.19}.
\end{proof}

\medskip
We conclude this section, by computing at any macroscopic time the
covariance of the process.  For $x_1\ne x_2\in \La_N$ and $s>0$ we call
		   \begin{equation}
   \label{1.30}
v_N(x_1,x_2;s) := \mathbb
E_{\eta_0}\Big(\prod_{i=1}^2\big[\eta_s(x_i)- \rho_N(x_i,s)\big]
\Big), 
\quad \rho_N(x_i,s)=\mathbb E_{\eta_0} (\eta_s(x_i))
   \end{equation}
We know from Corollary \ref{cor2} that $v_N$ converges to 0 as $N\to \infty$ and in the next Theorem we show that it is of order $N^{-1}$.
		


\begin{thm}
\label{a3.2}
For any initial configuration $\eta_0$, any $r_1\le r_2\in(-1,1)$ and any $t>0$
 we have 
 \begin{equation}
\lim_{N\to\infty}Nv_N\left( [Nr_1],[Nr_2] ;N^{2+\alpha}t\right)
=  - \frac 14 \big[\rho_+(t)-\rho_-(t)\big]^2  r_1(1-r_2)
  \label{a3.3}
   \end{equation}
 Furthermore for any $t>s$ 
    \begin{equation}
\lim_{N\to\infty}N \mathbb
E_{\eta_0}\Big(\big[\eta_{N^{2+\alpha}s}(x)- \rho_N(x,N^{2+\alpha}s)][\eta_{N^{2+\alpha}t}(y)- \rho_N(y,N^{2+\alpha}t)\big]
\Big)=0
  \label{a3.3b}
   \end{equation}
\end{thm}

\noindent{\bf Proof.} We compute the $t$-derivative of $v_N$:
	  \begin{eqnarray*}
  \nn
\frac {d v_N} {dt}&=&  \mathbb E_{\eta_0}\Big(\hat L_{N,t}\big[\prod_{i=1}^2(\eta(x_i,t)- \rho(x_i,t))\big]\Big)
\\&-&\sum_{i,j=1}^2 \mathbb E_{\eta_0}
      \Big(\frac{d\rho_N(x_i,t)}{dt}\big[\eta(x_j,t)-
      \rho(x_j,t)\big]\mathbf 1_{j\ne i}\Big) 
	 \label{1.33}
	  \end{eqnarray*}
After easy calculations we get:
	\begin{eqnarray}
	\nn
\frac {d v_N(x_1,x_2;t) } {dt} &=&L^{\rm st,N} v_N(x_1,x_2;t)-\frac 12 \big[\rho_N(x_1,t)-\rho_N(x_2,t)\big]^2 \mathbf 1_{|x_1-x_2|=1}
\\&-& \frac 12  v_N(x_1,x_2;t)\big[\mathbf 1_{|x_1|=N}+\mathbf 1_{|x_2|=N}\big]
	\label{1.34}
	\end{eqnarray}
where $v_N(x_1,x_2;t)$ is thought of as a function of the positions
$x_1$ and $x_2$ {and $L^{\rm st,N} $ is the generator
  defined in \eqref{genS} but with ``jumps" outside $\La_N$
  suppressed}. 
By setting $v_N(x_1,x_2;t)=0$ in the set 
$\dis{\bigcup_{i=1}^2\{|x_i|=N+1\}}$ we can rewrite \eqref{1.34} as an equation for $v_N(x_1,x_2;t)$ for  $x_1\ne x_2$ with $|x_i|\le N+1$ as follows
      \begin{eqnarray}
   \label{12.35}
&&\frac{d}{dt}v_N(x_1,x_2;t) = L^{\rm st} v_N(x_1,x_2;t) - \frac 12
\big[\rho_N(x_1,t)-\rho_N(x_2,t)\big]^2 \mathbf 1_{|x_1-x_2|=1}\nn\\
&& v_N(x_1,x_2;t)=0 \;\;\text{in $\bigcup_{i=1}^2\{|x_i|=N+1\}$ and when $t=0$}
   \end{eqnarray}
which
is solved by
      \begin{eqnarray}
&&v_N(x_1,x_2;t) = - \frac 12 \int_0^t E^{\rm st}_{x_1,x_2}\Big(
\big[\rho_N(x_1(s),t-s)-\rho_N(x_2(s),t-s)\big]^2
\nn
\\&& \hskip4cm\times 
 \mathbf 1_{|x_1(s)-x_2(s)|=1} \mathbf 1_{\tau_1>s, \tau_2>s} \Big)  \label{12.36}
   \end{eqnarray}
where $ E^{\rm st}_{x,y}$ denotes expectation with respect to the
process of two stirring particles, labelled 1 and 2,  which start from
$x_1\ne x_2\in\La_N$ respectively.   $\tau_1$ and $\tau_2$ are the
first time when particle $1$, respectively particle $2$, reaches
$\pm(N+1)$. From \eqref{1.11} we get that for $t_N=N^{2+\alpha}t${
     \begin{eqnarray}
\nn
&&\hskip-1cm\lim_{N\to\infty}\sup_{s\le N^{2+\alpha/2 }t}\,\,\sup_{x\in\La_N}N^2\Big|\big[\rho_N(x,t_N-s)-\rho_N(x+1,t_N-s)\big]^2
\\&&\hskip4cm-
\frac 12 \big[\rho_+(t)-\rho_-(t)\big]^2\Big|=0
   \label{a3.5}
   \end{eqnarray}}
From \eqref{1.12} we get {
    \begin{eqnarray}
 \nn 
&& \hskip-.8cm  \Big|\int_0^{t_N} E^{\rm st}_{x_1,x_2}\Big(
\big[\rho_N(x_1(s),t-s)-\rho_N(x_2(s),t-s)\big]^2\mathbf 1_{|x_1(s)-x_2(s)|=1} \mathbf 1_{\tau_1>s, \tau_2>s} \Big)
\\&&-\int_0^{N^{2+\alpha/2}} E^{\rm st}_{x_1,x_2}\Big(
\big[\rho_N(x_1(s),t-s)-\rho_N(x_2(s),t-s)\big]^2
\nn
\\&& 
\hskip1cm\times  
\mathbf 1_{|x_1(s)-x_2(s)|=1} \mathbf 1_{\tau_1>s, \tau_2>s} \Big)
\Big|\le c'e^{-c N^{\alpha/2}}
  \label{a3.6}
   \end{eqnarray} 
   }
From \eqref{a3.5} and \eqref{a3.6}, 
 \begin{eqnarray*}
&&\nn\hskip-1cm 
\lim_{N\to\infty}Nv_N\left( [Nr_1],[Nr_2] ;N^{2+\alpha}t\right)\\
&&=  - \frac 14 \big[\rho_+(t)-\rho_-(t)\big]^2 
 \int_0^\infty E_{r_1,r_2}\Big(\delta_{B(s)}\big(B'(s)\big)\mathbf
   1_{\tau>s, \tau'>s} \Big)
   \end{eqnarray*}
 $\delta_b(x)$ the Dirac delta at $x$, $B(s)$ and $B'(s)$ independent
 Brownian motions starting from $r_1$ and $r_2$ and $\tau$, $\tau'$
 the hitting times at $\pm 1$. As the right hand side of the above
 expression is identify to the kernel of $(-\Delta)^{-1}$ with
 Dirichlet boundary conditions, \eqref{a3.3} easily follows.

 To prove  \eqref{a3.3b} we observe that {
  \begin{eqnarray*}
&& \mathbb E_{\eta_0}\Big(\big[\eta_{N^{2+\alpha}s}(x)- \rho_N(x,N^{2+\alpha}s)][\eta_{N^{2+\alpha}t}(y)- \rho_N(y,N^{2+\alpha}t)\big]
\Big)
\\&&=\mathbb E_{\eta_0}\Big(\big[\eta_{N^{2+\alpha}s}(x)- \rho_N(x,N^{2+\alpha}s)]\eta_{N^{2+\alpha}t}(y)
\Big)
\\&&= \mathbb E_{\eta_0}\Big([\eta_{N^{2+\alpha}s}(x)-
          \rho_N(x,N^{2+\alpha}s)] 
\mathcal E_y(\mathbf 1_{\tau>N^{2+\alpha}(t-s)}\eta_{N^{2+\alpha}s}(x(N^{2+\alpha}(t-s))\Big)
\\&&\hskip.5cm+
\mathcal E_y(\mathbf 1_{\tau\le N^{2+\alpha}(t-s)}u_{x_\tau}(t-\tau))\mathbb E_{\eta_0}\Big(\big[\eta_{N^{2+\alpha}s}(x)- \rho_N(x,N^{2+\alpha}s)]\Big)
\\&&= \mathbb E_{\eta_0}\Big([\eta_{N^{2+\alpha}s}(x)-
          \rho_N(x,N^{2+\alpha}s)] 
\mathcal E_y(\mathbf 1_{\tau>N^{2+\alpha}(t-s)}\eta_{N^{2+\alpha}s}(x(N^{2+\alpha}(t-s))\Big)
   \end{eqnarray*}
} that by \eqref{1.12} implies  \eqref{a3.3b}.
 \qed

{As a corollary of  Theorem \ref{a3.2} one can prove that the fluctuation field converges to a Gaussian field as in the case of constants $\rho_\pm$ studied in \cite{LMO}. We do not proof this here.}

\section{Simple exclusion: entropy method}
\label{sec:entropy-method}

We prove here the quasi-static limit using the entropy
method. This gives a weaker result, but the methods extend to other
models where duality cannot be used. We assume here that
$\rho_{\pm}(t)$ are differentiable functions of time with bounded
derivative. This can be relaxed to Lipschitz continuity, as assumed in
the previous section, but the proof would be more involved.

In the following $\eta_t$ is the process generated by \eqref{0.1}. 

\begin{thmm}\label{leex}
  For any $\alpha >0$, any $t>0$, and any local function
  $\varphi(\eta)$,

\begin{equation}
  \label{eq:6}
  \lim_{N\to\infty} \mathbb E_{\eta_0} \left(\int_0^t ds \left|\frac 1N
      \sum_x G(\frac xN) 
     \theta_x\varphi(\eta_{sN^{2+\alpha}}) - \int_0^1 G(y)
       \hat\varphi(\bar\rho(y,s)) dy \right|\right) = 0.
\end{equation}
where $G$ is a continuous test function on $[-1,1]$.
\end{thmm}

The above theorem is stated for any initial configuration $\eta_0$. More generally we can start with any initial distribution, since in the following all we need is that the relative entropies are bounded by $CN$ for some constant $N$, that in this case is automatically satisfied by any probability measure on the configuration space $\{0,1\}^{2N+1}$.  

Consider the empirical distribution of the density at time $t$:
\begin{equation}
  \label{eq:7}
  \xi_{N,t}(G) = \frac 1{2N+1} \sum_{x=-N}^N G\left(\frac xN,t\right)
  \eta_{t}(x)  
\end{equation}
where $G(y,t)$ is a smooth function on $[-1,1]\times \bR_+$ with
compact support in 
$(-1,1)$. The time evolution is given by
\begin{equation}
  \label{eq:8}
  \xi_{N,t}(G) - \xi_{N,0}(G) = \int_0^t
  \xi_{N,s}((N^{\alpha} \partial_y^2 + \partial_s)G )
  ds  + O(N^{\alpha -1}) + M_N(t) 
\end{equation}
where $M_N(t)$ is a bounded martingale. By dividing by $N^{\alpha}$,
we have immediately that 
\begin{equation}
  \label{eq:9}
  \lim_{N\to\infty}  \int_0^t \xi_{N,s}(\partial_y^2 G)\; ds = 0.
\end{equation}
This implies that every limit point of the empirical distribution 
$\xi_{N,t}$ is a measure on $[-1,1]$ that satisfies Laplace equation,
in a weak sense,
with boundary conditions that we will identify in the following.

Corresponding to the quasi-stationary profile $\bar \rho(x,t)$ defined
in \eqref{eq:2}, we consider
the inhomogeneous product measure 
\begin{equation}
 \label{eq:3} 
 \mu_t(\eta)=\prod_{x=-N}^N\,  \bar \rho(\frac xN,t)^{\eta(x)}
 [1- \bar \rho(\frac xN,t)]^{1-\eta(x)}
\end{equation}
as reference measure at time $t$.
{The Dirichlet forms associated to the generator are then
  $\mathfrak{D}_{\pm N,t,\rho_\pm(t)}(f)$ and
  $\mathfrak{D}_{ex,t}(f)$: 
	\begin{eqnarray}
	\nn
&&\mathfrak{D}_{x,t,\rho}(f)
= \frac 12 \sum_\eta \rho^{1-\eta(x)}(1-\rho)^{\eta(x)} [\sqrt
    f(\eta^{x})-\sqrt f(\eta)]^2 \mu_t(\eta),\quad x=\pm N
    \\&&\mathfrak{D}_{ex,t}(f)=\frac 12\sum_\eta\sum_{x=-N}^{N-1}
\left(\nabla_{x,x+1}\sqrt{f(\eta)}\right)^2 \mu_t(\eta)
    	\label{dir1}
		\end{eqnarray}}
Let $f_{N,t}$ be such that the law of $\eta_t$ is given by 
$f_{N,t}\,\mu_t$, the following holds.
{	\begin{prop}
	\label{prop3.2}
There is  $C$ so that for all $t$
	\begin{equation}
  \label{eq:555}
    \int_0^t \left(\mathfrak{D}_{N,s,\rho_+(s)}(f_{N,s}) +
      \mathfrak{D}_{-N,s,\rho_-(s)}(f_{N,s}) +
      \mathfrak{D}_{ex,s}(f_{N,s})\right)\; ds  \le
 \frac{Ct}{N} .
\end{equation}
	\end{prop}}
	\begin{proof}
	We have that
	\begin{equation}
 \label{eq:4} 
\partial_t \big(f_{N,t}\,\mu_t\big)= (L_{N,t}^*f_{N,t})\,\mu_t
 \end{equation}
where $L_{N,t}^*$ is the adjoint  with respect to $\mu_t(\eta)$ that can be computed as: {
\begin{eqnarray}
\nn
&&\sum_{\eta} G(\eta)(L_{N,t}F)(\eta)\mu_t(\eta)= \sum_{\eta}
F(\eta)(L_{N,t}G)(\eta)\mu_t(\eta) \\
&&\nn+ N^{2+\alpha} 
\sum_{\eta}F(\eta)\sum_x G(\eta^{x,x+1}) \left(\mu_t(\eta^{x,x+1}) - \mu_t(\eta)\right)
 \end{eqnarray}
Observe that
\begin{eqnarray}
\nn
&&
 \left(\mu_t(\eta^{x,x+1}) - \mu_t(\eta)\right)
=
\Big[\big(\frac{\bar \rho(\frac
  {x+1}N,t)(1-\bar \rho(\frac {x}N,t)}{\bar \rho(\frac xN,t)(1-\bar
  \rho(\frac {x+1}N,t)}\big)^{\eta(x)-\eta(x+1)}-1\Big]\mu_t(\eta)\nn 
 \end{eqnarray}
and}
\begin{equation*}
  \begin{split}
    \big(\frac{\bar \rho(\frac {x+1}N,t)(1-\bar \rho(\frac
      {x}N,t)}{\bar \rho(\frac xN,t)(1-\bar \rho(\frac
      {x+1}N,t)}\big)^{\eta(x)-\eta(x+1)}-1= \big[1+\frac 1N
    \frac{\bar \rho'(\frac xN,t)}{\bar \rho(\frac xN,t)(1-\bar
      \rho(\frac xN,t)})\big]^{\eta(x)-\eta(x+1)}-1 \\
    =\frac 1N
    [\eta(x)-\eta(x+1)] \frac{\bar \rho'(\frac xN,t)}{\bar \rho(\frac
      xN,t)
      (1-\bar \rho(\frac xN,t))} {\ +\ O\left(\frac 1{N^2}\right)}\\
    =: \frac 1N [\eta(x)-\eta(x+1)] B\left(\frac xN,t \right) 
    {\ +\ O\left(\frac 1{N^2}\right)}. 
  \end{split}
\end{equation*}

%


Denote $H_N(t) = \sum_\eta f_{N,t}(\eta)\log f_{N,t}(\eta)\mu_t(\eta)$
\begin{equation*}
\begin{split}
\frac d{dt} H_N(t) =
  \sum_\eta \log f_{N,t}(\eta)(L_{N,t}^*f_{N,t})(\eta)\mu_t(\eta)
  +\sum_\eta [\partial_t f_{N,t}(\eta)]\mu_t(\eta)\\
  = \sum_\eta  f_{N,t}(\eta)(L_{N,t} \log f_{N,t})(\eta)\mu_t(\eta)
  +\sum_\eta [\partial_t f_{N,t}(\eta)]\mu_t(\eta)
\end{split}
\end{equation*}
Observe that, since 
$\frac d{dt} \sum_\eta f_{N,t}(\eta) \mu_t(\eta) = 0$, we have
  \begin{equation*}
    \sum_\eta [\partial_t f_{N,t}(\eta)]\mu_t(\eta) =-\sum_\eta
    f_{N,t}(\eta) [\partial_t\mu_t(\eta)] =O(N).
  \end{equation*}

By the inequality $a\log(b/a) \le 2\sqrt a (\sqrt b - \sqrt a)$, we
have that 
$$
f(\eta) (L_{N,t}\log f)(\eta) \le 2 \sqrt{f(\eta)} (L_{N,t}
\sqrt f)(\eta).
$$

Furthermore
   \begin{equation*}
     \begin{split}
       \sum_\eta 2 &\sqrt{f_{N,t}(\eta)} (L_{N,t} \sqrt f_{N,t})(\eta)
       \mu_t(\eta)\\
       =& - N^{2+\alpha} \sum_\eta\sum_{x=-N}^{N-1}
       \nabla_{x,x+1} \sqrt{f_{N,t}}(\eta)
       \nabla_{x,x+1}(\sqrt{f_{N,t}}\mu_t)(\eta) \\ 
       &- N^{2+\alpha} \left(\mathfrak{D}_{N,t,\rho_+(t)}(f_{N,t}) +
         \mathfrak{D}_{-N,t,\rho_-(t)}(f_{N,t})\right)
     \end{split}
\end{equation*}
and
\begin{equation*}
  \begin{split}
    \sum_\eta\sum_{x=-N}^{N-1} \nabla_{x,x+1}\sqrt{f_{N,t}}(\eta)
\nabla_{x,x+1}(\sqrt{f_{N,t}}\mu_t)(\eta)  
= \sum_\eta\sum_{x=-N}^{N-1}
\left(\nabla_{x,x+1}\sqrt{f_{N,t}}\right)^2 \mu_t  \\
+  \sum_\eta\sum_{x=-N}^{N-1} 
\sqrt{f_{N,t}}(\eta^{x,x+1}) \nabla_{x,x+1}(\sqrt{f_{N,t}})(\eta)
\nabla_{x,x+1}\mu_t(\eta)) \\ 
= \sum_\eta\sum_{x=-N}^{N-1}
\left(\nabla_{x,x+1}\sqrt{f_{N,t}}\right)^2 \mu_t   \\
+ \frac 1N \sum_\eta\sum_{x=-N}^{N-1} \sqrt{f_{N,t}}(\eta^{x,x+1}) 
\nabla_{x,x+1}(\sqrt{f_{N,t}})
 [\eta(x)-\eta(x+1)] B(\frac xN,t) \mu_t {\ +\ O(N^{-1})} \\
= 2\mathfrak{D}_{ex,t}(f_{N,t}) + \frac 1N \tilde B_N(t) {\
  +\ O(N^{-1})} 
  \end{split}
\end{equation*}
with
\begin{equation*}
\tilde B_N(t) = \sum_\eta\sum_{x=-N}^{N-1}
\sqrt{f_{N,t}}(\eta^{x,x+1})\nabla_{x,x+1}(\sqrt{f_{N,t}}) 
  \left(\eta(x)- \eta(x+1)\right) B(\frac xN,t)
 \mu_t 
\end{equation*}

By an elementary inequality we have: 
\begin{equation*}
  \begin{split}
    \left|\tilde B_N(t) \right| \le \frac N2
    \mathfrak{D}_{ex,t}(f_{N,t}) + \frac 2{2N} \sum_\eta\sum_{x=-N}^{N-1}
 {f_{N,t}}(\eta^{x,x+1})  \left(\eta(x)- \eta(x+1)\right)^2 B(\frac xN,t)^2
 \mu_t \\
 \le \frac N2
    \mathfrak{D}_{ex,t}(f_{N,t}) + \frac 2{2N}
   \sum_{x=-N}^{N-1} B(\frac xN,t)^2
    \sum_\eta {f_{N,t}} (\eta)
 \mu_t (\eta^{x,x+1}) \\
= \frac N2
    \mathfrak{D}_{ex,t}(f_{N,t}) + \frac 2{2N}
   \sum_{x=-N}^{N-1} B(\frac xN,t)^2
    \left( 1 + \sum_\eta {f_{N,t}} (\eta) \nabla_{x,x+1}\mu_t \right)
  \end{split}
\end{equation*}
and since $B( x,t) \le \frac 14 \|\bar\rho'\|_\infty$, iterating on
the bound for $\nabla_{x,x+1}\mu_t$, we obtain
\begin{equation*}
  \left|\tilde B_N(t)\right| \le \frac Na \mathfrak{D}_{ex,t}(f_{N,t}) +
   C a
\end{equation*}
All together we have
\begin{equation*}
  \begin{split}
    H_N(t) - H_N(0) = - N^{2+\alpha} \int_0^t
    \left(\mathfrak{D}_{N,s,\rho_+(s)}(f_{N,s}) +
      \mathfrak{D}_{-N,s,\rho_-(s)}(f_{N,s}) +
      \mathfrak{D}_{ex,s}(f_{N,s})\right)\; ds\\
    - N^{1+\alpha} \int_0^t \tilde B_N(s) \; ds + O(N^{1+\alpha}) t \\
    \le - N^{2+\alpha} \int_0^t \left(\mathfrak{D}_{N,\rho_+(s)}(f_{N,s}) +
      \mathfrak{D}_{-N,s,\rho_-(s)}(f_{N,s}) + 
      \mathfrak{D}_{ex,s}(f_{N,s})\right)\; ds\\
    + \frac{N^{2+\alpha}}{2} \int_0^t  \mathfrak{D}_{ex,s}(f_{N,s}) \; ds 
    + N^{1+\alpha} 2 Ct + O(N^{1+\alpha})t 
  \end{split}
\end{equation*}

\begin{equation*}
  \begin{split}
    \int_0^t \left(\mathfrak{D}_{N,\rho_+(s)}(f_{N,s}) +
      \mathfrak{D}_{-N,\rho_-(s)}(f_{N,s}) + \frac 12
      \mathfrak{D}_{ex}(f_{N,s})\right)\; ds \\
    \le \frac{H_N(0) + 
      N^{1+ \alpha} 2Ct + O(N^{1+\alpha})t}{N^{2+\alpha}} \le
    \frac{C''}{N^{1+\alpha}} + \frac{C'''t}{N} .
  \end{split}
\end{equation*}

\end{proof}

\begin{proof}[Proof of Theorem \ref{leex}]
Consider for any point $y\in [-1+k/N,1-k/N]$ the set $\Lambda_y$ of
   $2k+1$ integers defined by $\Lambda_y = \{[Ny]-k, \dots, [Ny]+k\}$,
   and define $f_{N,t}|_{\Lambda_y}$ the marginal of $f_{N,t}$ on $\Lambda_y$.
  It follows from \eqref{eq:555} that {for $\dis{
      f_{t,y}=\lim_{N\to\infty} f_{N,t}|_{\Lambda_y}}$} 
  \begin{equation*}
    \mathfrak{D}_{ex,k,t,y}(f_{t,y}) := \sum_{\eta\in\{0,1\}^{2k+1}} \sum_{x=-k}^{k-1}
    \left(\sqrt {f_{t,y}}(\eta^{x,x+1}) - \sqrt
      {f_{t,y}}(\eta)\right)^2 \mu_{\rho(y,t)}^{k} 
    = 0 
  \end{equation*}
  (where $\mu_{\rho(y,t)}^{k}$ is the Bernoulli measure on
  $\{0,1\}^{2k+1}$ with density $\rho(y,t)$). 
  This implies that $f_{t,y}(\eta_{-k},\dots,\eta_k)$ is symmetric for
  exchanges. Furthermore, considering the boundary blocks 
  $\{-N,\dots,-N+2k\}$ and 
  $\{N-2k, \dots, N\}$ we obtain respectively {
  \begin{equation*}
    \mathfrak{D}_{N,t,\rho_+(t)}(f_{t,-1}) = 0, \qquad 
      \mathfrak{D}_{-N,t,\rho_-(t)}(f_{t,1}) = 0
  \end{equation*}
  that implies $f_{t,1}$ and $f_{t,-1}$ are constant, and since they
  are probability densities with respect to $\mu_{\rho_+(t)}^{k}$ and
  $\mu_{\rho_-(t)}^{k}$, they are equal to 1 in those boundary blocks.}

From the above argument we obtain that 
\begin{equation}
  \label{eq:6k}
  \begin{split}
    &\lim_{k\to\infty} \lim_{N\to\infty} \int_0^t ds \int_{-1+k/N}^{1-k/N} dy  \sum_\eta f_{N,s}(\eta) \mu_s(\eta)\\
    &\left|\frac 1{2k} \sum_{|x-Ny|\le k}
      \theta_x\varphi(\eta_{s}) - \hat\varphi\left(\frac 1{2k}
      \sum_{|x-Ny|\le k} \eta(x)\right) \right|  =
    0.
  \end{split}
\end{equation}
Next we extend this statement to macroscopic blocks. 
Consider now, for any $y,y'\in[-1+ k/N, 1- k/N]$ the two blocks
$\Lambda_y$ and 
$\Lambda_{y'}$, and let $f_{N,t,y,y'}(\eta,\tilde\eta)$ the
corresponding joint marginal. 
Define, for functions on two separate blocks $f(\eta_{-k}, \dots,
\eta_k; \tilde \eta_{-k}, \dots,\tilde \eta_k)$, the Dirichlet form
corresponding to the exchange of the occupation of the centers of the
box:
\begin{equation*}
  \mathfrak{D_0} (f) = \sum_{\eta,\tilde\eta} \left( \sqrt f
    (T_0(\eta,\tilde\eta)) - 
    \sqrt f(\eta,\tilde\eta) \right)^2 \mu_{\bar\rho(t,y)}^{k} (\eta)
  \mu_{\bar\rho(t,y')}^{k}(\tilde\eta) 
\end{equation*}
where $T_0$ is the exchange of $\eta(0)$ with $\tilde\eta(0)$. 

By the same telescoping argument used in \cite{KOV1989}, we have that
{ $f_{t,y,y'}(\eta,\tilde\eta)=\lim f_{N,t,y,y'}(\eta,\tilde\eta)$, satisfies}
\begin{equation*}
   \mathfrak{D_0} (f_{t,y,y'}) \le C|y-y'|^2
\end{equation*}

It follows that (see \cite{KLbook} for standard details):
  \begin{equation}\label{eq:scales}
    \begin{split}
      \lim_{\eps\to 0}& \lim_{k\to\infty} \lim_{N\to\infty} \int_0^t ds
      \int_{-1+k/N}^{1-k/N} dy \sum_\eta  f_{N,s}(\eta) \mu_s(\eta) \\
      &\left( \frac 1{2k+1} \sum_{x=[Ny]-k}^{[Ny]+k} \eta(x) - \frac
        1{2N\eps}\sum_{x=[N(y-\eps)]}^{[N(y+\eps)]}\eta(x) \right)^2
     = 0
    \end{split}
  \end{equation}
In fact we already know that
\begin{equation}
  \label{eq:101}
  \lim_{k\to\infty} \lim_{N\to\infty} \int_0^t ds 
  \sum_\eta \left(\frac 1{2k+1} \sum_{x=-N}^{-N+2k} \eta(x) -
    \rho_-(s)\right)^2  f_{N,s}(\eta) \mu_s(\eta) =0
\end{equation}
and similarly for the block $[N-2k,N]$. By \eqref{eq:scales} we have 
\begin{equation}
  \label{eq:102}
  \lim_{\eps\to 0} \lim_{N\to\infty} \int_0^t ds 
  \sum_\eta \left(\frac 1{2N\eps} \sum_{x=-N}^{-N(1+2\eps)} \eta(x) -
    \rho_-(s)\right)^2  f_{N,s}(\eta) \mu_s(\eta) =0
\end{equation}
This implies, with \eqref{eq:9}, that the limit of $\frac
        1{2N\eps}\sum_{x=[N(y-\eps)]}^{[N(y+\eps)]}\eta(x)$ as
        $N\to\infty$ and $\eps\to 0$ converges strongly to the
        solution of the 
        Laplace equation with this boundary condition given by
        $\rho_{\pm}(t)$, i.e. $\bar\rho(y,t)$.

\end{proof}

\section{Zero range}
\label{sec:zero-range}

The results of the previous section extend quite straightforwardly to
any \emph{gradient} conservative dynamic. As an example, let us
consider the zero range process, whose generator is given by 
\begin{equation}
  \label{eq:111}
  \begin{split}
    L^{\text{0r}}_{N,t} f (\eta) = \frac{N^{2+\alpha}}2 \Big\{ &\sum_{x=-N}^{N} g(\eta(z))
    \sum_{\sigma = \pm 1} \left[ f(\eta^{x,x + \sigma}) - f(\eta)
    \right]\\
    & + \sum_{\sigma = \pm} \lambda_{\sigma}(t) [f(\eta^{\sigma
      N,+})-f(\eta)] \Big\}
  \end{split}
\end{equation}
where $\eta^{x,y}$ is the configuration $\eta$ with a particle moved
from $x$ to $y$, for $x,y= -N, \dots, N$. For $\eta^{N,N+1}$ and
$\eta^{-N,-N-1}$ we have destroyed a particle in the correspondig
site, while $\eta^{x,+}$ means we have addeed a particle in the site
$x=N, -N$.
 We assume that the rate function $g:\bN \to \bR_+$ such that $g(0) = 0$,
$g(k)>0$ for $k>0$, and $\sup_{k} |g(k+1) - g(k)| < +\infty $.
As before we assume that $\lambda_\pm(t)$ are differentiable 
functions of time with bounded derivative.

{For any $\la>0$ consider the measure $ \mu_\lambda$ on the non negative integer
\begin{equation*}
  \mu_\lambda(k) = \frac{\lambda^k}{g(k)!} \frac 1{Z(\lambda)}, \qquad g(k)! =
  g(1)\dots g(k), \ g(0)! = 1.  
\end{equation*}
where $Z(\lambda)$ is the normalization constant.

Let $\bar\lambda(r,t)$, $r\in[-1,1]$ be  the linear 
interpolation of $\la_\pm(t)$ defined as in \eqref{eq:2}. 

Consider as reference measure the inhomogeneous product
 \begin{equation}
  \mu^N_t(\eta)=\prod_{x=-N}^N\, \mu_{\bar\lambda(\frac xN,t)}({\eta(x)}) 
 \end{equation}
 Observe that  for all $N$,  $t$ and all local function $\varphi$:
\begin{equation}
\label{4.4a}
  \sum_\eta   L^{\text{0r}}_{N,t} \varphi (\eta) \mu^N_t(\eta) \ =\ 0.
\end{equation}

For constant $\lambda$, we denote by $\mu_\lambda$ the corresponding
homogeneoous product measure. 

Assume that the initial configuration $\eta_0$ is randomly distributed by 
a probability measure with density $f_{N,0}$ with respect to $\mu^N_0(\eta)$ and such that the relative entropy $\sum_\eta f_{N,0}(\eta) \log f_{N,0} (\eta)\mu^N_0(\eta) \le CN$ for some constant $C>0$.  

\begin{thm}
\label{lezeror}

  For any local function $\varphi$, denoting $\hat\varphi(\lambda) =
<\varphi>_\lambda$, the average of $\varphi$ with respect to $ \mu_\lambda$:
\begin{equation}
  \label{eq:66}
  \lim_{N\to\infty} \mathbb E \left(\int_0^t ds \left|\frac 1N
      \sum_x G(\frac xN) 
     \theta_x\varphi(\eta_{s}) - \int_0^1 G(y)
       \hat\varphi(\bar\lambda(y,s)) dy \right|\right) = 0.
\end{equation}
where $G$ is a measurable bounded test function and $\theta_i$ is
the space shift by $i$, well defined for $N$ large enough. 
\end{thm}

We omit the  proof of Theorem \ref{lezeror}  being very similar to
the one of Theorem \ref{leex} for simple exclusion.

As observed in \cite{DF} in the case of $\la_\pm$ constants, \eqref{4.4a} straightforwardly yields \eqref{eq:66}.
}



\section{Damped anharmonic chain in temperature gradient}
\label{sec:damp-anharm-chain}

We consider a chain of $N$ coupled oscillators in one dimension. Each
particle has the same mass, equal to one. The configuration in the
phase space is described by $\eta = \{q_x,p_x, x = 1, \dots,N\}\in \bR^{2N}$.
The interaction between two particles $x$ and $x-1$ is described by
the potential energy $V(q_x-q_{x-1})$ of an anharmonic spring. The
chain is attached on the left to a fixed point, so we set $q_0(t) =0,
p_0(t)=0$.  
 We call $\{r_x=q_x-q_{x-1}, x=1,\dots, N\}$ the interparticle
 distance. We assume V to be a positive smooth function,
and that there exists a constant $C>0$ such that:
\begin{equation}\label{eq:V}
\begin{split}
\lim_{|r|\to \infty} \frac{V(r)}{|r|}=\infty, \quad
\lim_{|r|\to \infty} V''(r)\leq C<\infty
\end{split}
\end{equation}

Energy is defined by the following Hamiltonian:
\begin{equation}
\cH
:=\sum_{x=1}^N\left(\frac{p_x^2}{2}+V(r_x)\right) 
\end{equation}

The particle dynamics is subject to an interaction with an environment
given by Langevin heat bath at different temperatures $\beta_x^{-1}$. 
We choose $\beta_x$ as slowly varying on a macroscopic scale,
i.e. $\beta_x = \beta(x/N)$ for a given smooth strictly positive function
$\beta(y)$, $y\in [0,1]$.

The equations of motion are given by
\begin{equation}
\begin{cases}
dr_x(t)=&N^{2+\alpha}(p_x(t)-p_{x-1}(t))dt\\
dp_x(t)=&N^{2+\alpha}(V'(r_{x+1}(t))-V'(r_x(t)))\; dt \\
&-N^{2+\alpha} \gamma p_x(t) dt +
N^{1+\alpha/2}\sqrt{\frac{2\gamma}{\beta_x}}dw_x(t), \quad x=1,..,N-1\\
dp_N(t)= &N^{2+\alpha}(\bar{\tau}(t)-V'(r_N(t)))\;dt-N^{2+\alpha}\gamma
p_N(t)\; dt \\
&+N^{1+\alpha/2}\sqrt{\frac{2\gamma}{\beta_N}} dw_N(t) .
\end{cases}
\end{equation}
Here $\{w_x(t)\}_x$ are $N$-independent Wiener processes, $\gamma>0$ is
the coupling parameter with the Langevin thermostats.  The force
$\bar\tau(t)$ is assumed to be a smooth function of the
\emph{macroscopic} time $t$. 
 
The generator of the process is given by
\begin{equation}
\cL^{\bar{\tau}(t)}_n:= N^{2+\alpha}\left(\cA^{\bar{\tau}(t)}_N+ \gamma \cS_N\right),
\end{equation}
where $\cA^{\bar{\tau}}_N$ is the Liouville generator 
\begin{equation}
A_N^{\bar{\tau}}=\sum_{x=1}^N (p_x-p_{x-1})\partial_{r_N}
+\sum_{x=1}^{N-1}(V'(r_{x+1})-V'(r_x))\partial_{p_x}+(\bar{\tau}-V'(r_N))\partial_{p_N} 
\end{equation}
while $\cS_N$ is the operator 
\begin{equation}
\cS_N=\sum_{x=1}^N\left(\beta^{-1}_x\partial_{p_x}^2-p_x\partial_{p_x}\right)
\end{equation}

An equivalent dynamics is given by a different modelling of the heat
bath: particle $x$ undergoes stochastic elastic collisions with
\emph{particles of the environment} at temperature $\beta^{-1}_x$,
i.e. at independent 
exponentially distributed times of intensity $\gamma N^{2+\alpha}$ particle $x$
changes its velocity to a new random velocity normally distributed with
variance $\beta^{-1}_x$. The evolution equations are given by: 
\begin{equation}\label{eq:intcoll}
\begin{cases}
dr_x(t)=&N^{2+\alpha}(p_x(t)-p_{x-1}(t))dt\\
dp_x(t)=&N^{2+\alpha} (V'(r_{x+1}(t))-V'(r_x(t)))dt \\
& + \left(\tilde
  p_{x,{\mathcal N}_x(\gamma N^{2+\alpha} t)} - p_x(t^-)\right) d{\mathcal N}_x(\gamma
 N^{2+\alpha} t), \qquad x=1,..,N-1\\
dp_N(t)=&N^{2+\alpha}(\bar{\tau}(t)-V'(r_N))dt \\
 & + \left(\tilde
  p_{N,{\mathcal N}_N(\gamma N^{2+\alpha} t)} - p_N(t^-)\right) 
d{\mathcal N}_N(\gamma N^{2+\alpha} t) 
\end{cases}
\end{equation}
where $\tilde p_{x,k}$ are independent gaussian variables on mean zero
and variance $\beta^{-1}_x$, and $\{{\mathcal N}_x(t), x=1,\dots,N\}$ are
independent Poisson processes of intensity 1.

For $\bar{\tau}(t)=\tau$ constant, and $\beta_x=\beta$ homogeneous,
the system has a unique invariant 
measure given by a product of invariant Gibbs
measures $\mu_{\tau,\beta}^N$:
\begin{equation}\label{eq:gibbscan}
d\mu_{\tau,\beta}^N =\prod_{x=1}^N e^{-\beta(\mathcal{E}_x-\tau
  r_x)-\mathcal{G}_{\tau,\beta}} dr_xdp_x 
\end{equation}
where $\cE_x$ is the energy of the particle $x$:
\begin{equation}
\mathcal{E}_x = \frac{p_x^2}{2}+V(r_x).
\end{equation}

The 
function $\mathcal{G}_{\tau,\beta}$ is the Gibbs potential defined as:
\begin{equation}
\mathcal{G}_{\tau,\beta}=\log{\left[\sqrt{2\pi\beta^{-1}}\int
    e^{-\beta(V(r)-\tau r)} dr \right]}.
\end{equation}

The free energy of the equilibrium state $(r,\beta)$ is given by the
Legendre transform of $-\beta^{-1}\cG_{\tau,\beta}$:
\begin{equation}
\cF_{r,\beta}=\sup_{\tau}\{ \tau r + \beta^{-1}\cG_{\tau,\beta}\}
\end{equation}
The corresponding convex conjugate variables are the lenght
\begin{equation}\label{grandpot}
\mathfrak{r}(\tau,\beta)=\beta^{-1}\partial_\tau \cG_{\tau,\beta}
\end{equation} 
and the tension 
\begin{equation}
\bm{\tau}(r,\beta)=\partial_r \cF_{r,\beta}.
\end{equation} 
Observe that 
\begin{equation}
\mathbb{E}_{\mu_{\tau,\beta}^N}[r]=\mathfrak{r},
\qquad
\mathbb{E}_{\mu_{\tau,\beta}^N}[V'(r)]=\bm{\tau}.
\end{equation}
\\
Let $\nu^N_{\beta_\cdot}$ the inhomogeneous Gibbs measure
\begin{equation}
d\nu^N_{\beta_\cdot}= \prod_{x=1}^N \frac{e^{-\beta_x\cE_x}}{{Z_{\beta_x}}}
\end{equation}
Observe that this is \textbf{not} the stationary measure for this
dynamics for $\bar\tau = 0$.

Also we will use the thermodynamic entropy defined as
\begin{equation}
  \label{eq:tentropy}
  S(u,r) = \inf_{\beta>0} \left\{ - \beta u - \beta \mathcal F(r,\beta) \right\}.
\end{equation}

Let $f^N_t$ the density, with respect to $\nu^N_{\beta_\cdot}$, of the
probability distribution of the system at time t, i.e. the solution of
\begin{equation}
  \label{eq:fk}
  \partial_t f^N_t = \mathcal L_N^{\bar\tau(t),*} f^N_t ,
\end{equation}
where $\mathcal L_N^{\bar\tau(t),*}$ is the adjoint of 
$\mathcal L_N^{\bar\tau(t)}$ with respect to $\nu^N_{\beta_\cdot}$,
i.e. explicitely 
\begin{equation}
  \begin{split}
    (\cL_N^{\bar{\tau}(t)})^*= - N^{2+\alpha} \cA_N^{\tau(t)} -
    N^{2+\alpha}\sum_{x=1}^{N-1} (\beta_{x+1} - \beta_x)
    p_xV'(r_{x+1})\\ + N^{2+\alpha}\beta(1) p_x\bar{\tau} (t)
    + N^{2+\alpha}\gamma \cS_N.
  \end{split}
\label{eq:adjointop}
\end{equation}

Define the relative entropy of $f^N_t d\nu^N_{\beta_\cdot}$ with respect
to $d\nu^N_{\beta_\cdot}$:
\begin{equation}
H_N (t) = \int f^N_t\log{f_t^N} d\nu^N_{\beta_\cdot}.
\end{equation}
We assume that the initial density $f^N_0$ satisfy the bound
\begin{equation}
  \label{eq:31}
  H_N (0) \le CN.
\end{equation}
We also need some regularity of  $f^N_0$: define the hypercoercive
Fisher information functional:
\begin{equation}
  \label{eq:hfi}
  \begin{split}
    I_N(t) &= \sum_{x=1}^{N-1} \beta^{-1}_x \int
    \frac{\left( \partial_{p_x} f^N_t
        + \partial_{q_x} f^N_t\right)^2}{f^N_t} 
   d\nu_{\beta\cdot} 
  \end{split}
\end{equation}
where $\partial_{q_x} = \partial_{r_{x}} - \partial_{r_{x+1}},
x=1,\dots,N-1$, and $\nu_{\beta\cdot}:=\nu^N_{\beta\cdot}$. 
We assume that
\begin{equation}
  \label{eq:32}
  I_N(0) \le K_N 
\end{equation}
with $K_N$ growing less than exponential in $N$.

Consider a local function $\varphi(\eta)$ such that 
\begin{equation}
|\varphi(\eta) |\le C \sum_{x\in\Lambda_\varphi} (p_x^2 +
V(r_x))^\delta, \qquad \delta < 1 \label{eq:bound}
\end{equation}
where $\Lambda_\phi$ is the local support of $\varphi$.
Let $\theta_i \varphi$ be the shifted function (well defined for large
enough $N$).  Denoting 
$\hat\varphi(\tau,\beta) =<\varphi>_{\tau,\beta}$ the expectation with
respect to $d\mu^N_{\tau\beta}$ defined by \eqref{eq:gibbscan} 

\begin{prop}\label{leoscill}
\begin{equation}
  \label{eq:666}
  \lim_{N\to\infty} \mathbb E^N \left(\int_0^t ds \left|\frac 1N
      \sum_x G(\frac xN) 
     \theta_x\varphi(\eta_{s}) - \int_0^1 G(y)
       \hat\varphi(\bar\tau(s),\beta(y)) dy \right|\right) = 0.
\end{equation}
\end{prop}

The proof uses similar ideas as in the previous cases, plus results
contained in \cite{LO2015}. In particular from same argument used in
\cite{LO2015} follows that the empirical distribution for the $r$'s
converge to the solution of 
\begin{equation}
  \label{eq:oscw}
  \int_0^t ds \left[\int_{0}^1 \partial_y^2 G(y,s)
    \bm{\tau}(r(y,s),\beta(y)) dy - \partial_y G(1,s) \bar\tau(s)
  \right] =0
\end{equation}
for any smooth function $G(y,t)$ on $[0,1]$ such that $G(1)=0$ and $G'(0)=0$.

\subsection{Some ``thermodynamic'' consequences }
\label{sec:some-therm-cons}

We illustrate here how the above limit realizes the so-called
\emph{quasi-static} isothermal trasformation of the
thermodynamics. When a temperature gradient is present, 'isothermal'
should be intended that the temperature gradient does not change in
time. In this case these are quasistatic transformations between
non-equilibrium stationary states.  

When performing the usual diffusive scaling, as done in
\cite{olla2014} for constant temperatures and in \cite{LO2015} for the
temperature gradient case, the results below are obtained in a two
step limit, first the hydrodunamic diffusive limit, then a
quasi-static limit, see details in \cite{olla2014,LO2015}

\subsubsection{Excess Heat}
\label{sec:excess-heat}

The (normalized) total internal energy of the system is defined
by 
\begin{equation}
  \label{eq:14}
  U_N := \frac 1N \sum_{x=1}^N\left(\frac{p_x^2}{2}+V(r_x)\right) 
\end{equation}
then internal energy evolves as:
\begin{equation*}
   U_N(t) -  U_N(0) = \mathcal W_N(t) + Q_N(t)
\end{equation*}
where 
\begin{equation*}
  \mathcal W_N(t) = N^{1+\alpha} \int_0^t \bar\tau(s) p_N(s) ds = \int_0^t
  \bar\tau(s) \frac{dq_N(s)}n 
\end{equation*}
is the (normalized) work done by the force $\bar\tau(s)$ up to time
$t$, 
while
\begin{equation}\label{eq:heat}
  \begin{split}
    Q_N(t) = &\gamma\; N^{1+\alpha} \sum_{x=1}^N \int_0^t ds
    \left(p_x^2(s) - \beta_x^{-1}\right) \\
   & + N^{\alpha/2} \sum_{x=1}^N
    \sqrt{2\gamma \beta^{-1}_x} \int_0^t p_x(s) dw_x(s).
  \end{split}
\end{equation}
is the total flux of energy between the system and the heat bath
(divided by $N$). 
As a consequence of theorem \autoref{leoscill} we have that 
\begin{equation*}
  \lim_{N\to\infty} \mathcal W_N(t) = \int_0^t \bar\tau(s) d\mathcal
  L(s) := \mathcal W(t)
\end{equation*}
  where $\mathcal L(t) = \int_0^1 r(y,t) dy$, the total macroscopic
  length at time $t$. While for the energy difference we expect that
  \begin{equation}\label{eq:unproven}
    \begin{split}
      \lim_{n\to\infty} \left(U_N(t) - U_N(0)\right) &= \int_0^1  
      \left[u({\bar \tau}(t),\beta(y)) - 
        u(\bar\tau(0), \beta(y))\right] dy
    \end{split}
  \end{equation}
where $u(\tau,\beta)$ is the average energy for
$\mu_{\beta,\tau}$, i.e.
\begin{equation*}
  u(\tau,\beta) = \int \mathcal E_1 d\mu^1_{\tau, \beta} = 
  \frac 1{2\beta} + \int V(r) e^{-\beta (V(r) - \tau r) -
    \mathcal{\tilde G}(\tau,\beta)} dr 
\end{equation*}
 with $\mathcal{\tilde G}(\tau,\beta) = \log \int e^{-\beta (V(r) -
 \tau r)} dr$. 
Unfortunately for lack of uniform integrability for the energy
distribution,  we do not have a rigorous proof of
\eqref{eq:unproven}, since energy correspond to a value $\delta = 2$
in \eqref{eq:bound}. 
Assuming that the local equilibrium established in \autoref{leoscill}
extends to quadratic growing functions, then we have obtained
\begin{equation}
  \label{eq:exheat}
  Q_N(t) \mathop{\longrightarrow}_{N\to\infty} Q(t)
\end{equation}
where $Q(t)$ is deterministic and satisfy the relation
\begin{equation}
  \label{eq:1st-pr}
  Q(t) = \int_0^1 \left[u({\bar \tau}(t),\beta(y)) - u(\bar\tau(0),
    \beta(y))\right] dy - \mathcal W(t).
\end{equation}
We call $Q(t)$ excess heat (or heat in the case of constant profile of
temperature), and \eqref{eq:1st-pr} express the first principle of
thermodynamics for quasistatic transformations.

Notice that from \eqref{eq:heat}, $Q_N(t)$ is the time integral of a
highly fluctuating random quantity. It is only after the particular
space-time scaling that the this quantity converges, and that the limit is a
\emph{deterministic} function of time. In this sense this result
is different from similar identifications of \emph{heat} as a \emph{stochastic}
flux of energy done in the so-called ``Stochastic Thermodynamics''
(cf. \cite{Seifert} for example). 
 On the other hand in
thermodynamics \emph{heat} is 
defined as the total flux of energy between the system and the thermal
bath, resulting from the complete transformation from one stationary
state to another, while it does not attempt to describe the
instantaneous flux of energy. But in a quasi-static transformations at
each time $t$ a new stationary state is reached.

\subsubsection{Free energy and Clausius identity}
\label{sec:free-energy}

Define
\begin{equation}
  \label{eq:16}
  \widetilde{\mathcal F}(t) = \int_0^1 \mathcal F(r(y,t), \beta(y)) dx
\end{equation}
A straightforward calculation gives
\begin{equation}
  \label{eq:clausius}
  \begin{split}
    \widetilde{\mathcal F}(t) -  \widetilde{\mathcal F}(0) = 
   \int_0^t ds \int_0^1 dy \bar\tau(s) \partial_s r(y,s) 
= \int_0^t \bar\tau(s) d\mathcal L(s) =  \mathcal W(t) 
  \end{split}
\end{equation}
  i.e. Clausius equality for the free energy.

Equivalently, by using the thermodynamic relation $\mathcal F = u - \beta^{-1}
S$ from \eqref{eq:tentropy}, we have 
\begin{equation}
  \label{eq:2nd}
  \int_0^1 dy\ \beta^{-1}(y) \left( S(r(y,t), u(y,t)) -  S(r(y,0),
    u(y,0))\right) = Q(t)
\end{equation}
In the case of constant temperature profile, this reduce to the
expected thermodynamic relation $\dot S = \beta \dot Q$ for
quasistatic isothermal thermodynamic trasformations.

\bibliographystyle{amsalpha}

\end{document}